\documentclass[11pt]{amsart}
\usepackage{amsmath,amssymb,amsfonts}
\usepackage{amsthm}
\usepackage{mathrsfs}
\usepackage[margin=1in]{geometry}

\theoremstyle{plain}
\newtheorem{theorem}{Theorem}
\newtheorem{proposition}[theorem]{Proposition}
\theoremstyle{definition}
\newtheorem{example}{Example}
\newtheorem{remark}{Remark}
\newtheorem{corollary}{Corollary}

\newtheorem{lemma}[theorem]{Lemma}

\newcommand{\R}{\mathbb{R}}

\newcommand{\tr}{\mathrm{tr}}
\newcommand{\dom}{\operatorname{dom}}

\newcommand{\rank}{\operatorname{rank}}
\begin{document}

	\title[Hidden Convexity of Three Quadratic Forms]{Hidden Convexity of the Joint Range of Three Quadratic Forms and Applications}

\author{Hoang Ngoc Tuan}
\address{Department of Mathematics, Hanoi Pedagogical University 2, Phu Tho, Vietnam}
\email{hoangngoctuan@hpu2.edu.vn}
\author{Nguyen Quang Huy}
\address{Department of Mathematics, Hanoi Pedagogical University 2, Phu Tho, Vietnam}
\email{nqhuy@hpu2.edu.vn}

	\begin{abstract}
	We investigate a hidden convexity phenomenon for triples of homogeneous
		quadratic forms. We give a complete characterization of the convexity of
		the upper image obtained by adding the nonnegative orthant to their joint
		range. More precisely, the upper image is convex if and only if either a
		nonzero nonnegative linear combination of the three associated matrices is
		positive semidefinite, or the three quadratic forms admit a common strict
		negative direction. In particular, the upper image is convex whenever one
		of the three quadratic forms is convex. As applications, we establish the
		convexity of an associated value function and derive a generalized S-lemma
		for implications involving two quadratic inequalities and a strict theorem
		 of the alternative for quadratic systems.
	\end{abstract}

	\keywords{Hidden convexity, Quadratic forms, Joint range, Upper image, S-lemma}

	\subjclass[2020]{90C20, 52A41, 15A63, 90C26}

	\maketitle
	
\section{Introduction}

Quadratic mappings and sets described by quadratic inequalities are widespread in optimization, control, and
signal processing \cite{AiLiangYuan2025,BeckEldar2006,BomzeJeyakumarLi2018,BazanOpazo2021,FloudasVisweswaran1995,JeyakumarHuyLi2009,LTY05book,LuoMaSo2010,Yakubovich1977}. Despite their simplicity, quadratic models often generate
highly nonconvex feasible regions and images. A remarkable and useful phenomenon, however, is that
many such nonconvex objects exhibit \emph{hidden convexity}: after a suitable reformulation,
the underlying structure becomes convex and can be exploited both algorithmically and theoretically.

Let $f_i$, $i=1,\ldots,m$, be quadratic functions on $\mathbb{R}^n$.
A variety of fundamental statements for quadratic systems depend crucially on the convexity of the upper image of the joint range
$\{(f_1(x), \ldots, f_m(x)) : x\in \mathbb{R}^n\}$, defined by
$${\rm U}(f_1,\ldots,f_m):=
\{(f_1(x), \ldots, f_m(x)) : x\in \mathbb{R}^n\} + \mathbb{R}^{m}_+.$$
Observe that ${\rm U}(f_1,\ldots,f_m)$ is convex whenever the joint range is convex.

Geometrically, the convexity of ${\rm U}(f_1,\ldots,f_m)$ allows nonlinear quadratic systems
to be treated through convex separation arguments, which lie at the core of duality,
S-procedures, and alternative theorems.
This geometric viewpoint explains why convexity of the upper image naturally leads to powerful analytical consequences.

In particular, the convexity of ${\rm U}(f_1,\ldots,f_m)$ is a key ingredient in the analysis of exact relaxations and strong duality \cite{BeckEldar2006,JeyakumarLi2014}.
Moreover, \cite[Theorem~1]{LaraUrruty2022} shows that the S-procedure holds under the convexity of ${\rm U}(f_1,\ldots,f_m)$ together with Slater's condition.
Furthermore, this convexity property also plays an important role in alternative theorems and minimax formulations \cite{Roubi2022}.

Despite its importance, the convexity of ${\rm U}(f_1,\ldots,f_m)$ does not hold in general.
It holds for $m=2$ \cite{TuyTuan2013,BazanOpazo2016,HHNTT2025}, but already fails for $m=3$ \cite[Example~2.1]{JeyakumarLi2014}.
Consequently, convexity may fail for any $m>2$, since  the case $m=3$ can be embedded into higher dimensions by appending zero quadratic forms.
For general $m$, the authors in \cite{JeyakumarLi2014} establish the convexity of ${\rm U}(f_1,\ldots,f_m)$ only for a particular trust-region-type setting, under an additional dimension condition imposed on the original data.

We now focus on the setting of (homogeneous) quadratic forms. This focus is natural, since optimization problems involving quadratic-form data have many important applications;
see \cite{AiLiangYuan2025,LuoMaSo2010} and the references therein.
It is well known that the joint range of two quadratic forms is convex (indeed, a convex cone) \cite{Dines1941,PolikTerlaky2007}.
However, the joint range of more than two quadratic forms need not be convex; for instance, \cite[Example~3.2]{Polyak1998} provides a counterexample for $m=3$.

Nevertheless, for the same data as in \cite[Example~3.2]{Polyak1998}, the corresponding upper image is still convex; see Example~\ref{exam_convexcase}  for details.
Note that in this example one component of the quadratic form system is convex, though not strongly convex.
In addition, for $n\ge3$, it follows from \cite[Theorem~2.1]{Polyak1998}  that if at least one of the quadratic forms $f_1,f_2,f_3$ is strongly convex, then the joint range
$\{(f_1(x),f_2(x),f_3(x)):\ x\in\mathbb R^n\}$ is convex; hence, the upper image ${\rm U}(f_1,f_2,f_3)$ is also convex.
Finally, Example~\ref{exam_nonconvex} demonstrates that, in general, if none of the components is convex, then ${\rm U}(f_1,f_2,f_3)$ need not be convex.

Motivated by the above discussion, this paper investigates the convexity of the upper image in the case $m=3$ within the homogeneous setting of quadratic forms.
More concretely, let $A_1,A_2,A_3\in\mathbb R^{n\times n}$ be symmetric matrices, and consider the homogeneous quadratic forms
$f_i(x)=\tfrac12 x^\top A_i x$ for $i=1,2,3$.
We define
$$
S:={\rm U}(f_1,f_2,f_3)= \{(f_1(x),f_2(x),f_3(x)):\ x\in\mathbb R^n\}+\mathbb R^3_+ .
$$

The main result of this paper gives a complete characterization of the
convexity of $S$. We prove that $S$ is convex if and only if at least one
of the following two conditions holds: either there exists
$\lambda\in\mathbb R_+^3\setminus\{0\}$ such that
$$
\lambda_1A_1+\lambda_2A_2+\lambda_3A_3\succeq0,
$$
or there exists $d\in\mathbb R^n$ such that
$$
d^\top A_i d<0,\qquad i=1,2,3.
$$
The first condition includes, as a special case, the situation where one
of the three quadratic forms is convex. Under the second condition, the
upper image is in fact equal to $\mathbb R^3$.

A key step in the proof is Lemma~\ref{lem_rankone}, which provides the
required rank-one reduction under the first condition. As consequences of
the characterization, we also obtain the convexity of an associated value
function and derive applications to a generalized S-lemma and to a theorem
of the alternative.

The remainder of the paper is organized as follows. Section~2 presents preliminaries. Section~3 proves the main characterization of the convexity of $S$ and gives
several examples. Section~4 presents applications of the convexity results, including the convexity of an associated value function, a generalized S-lemma, and a theorem of the alternative. Concluding remarks are given in the final section.

\section{Preliminaries}\label{Preliminaries}

We equip $\R^n$ with the Euclidean inner product
$\langle x,y\rangle:=x^\top y$ and the induced norm
$\|x\|:=\sqrt{x^\top x}$. For $m\in\mathbb N$, $\R_+^m$
denotes the nonnegative orthant.

Let $\mathbb S^n$ be the space of real symmetric $n\times n$ matrices and let
$\mathbb S_+^n$ denote the cone of positive semidefinite matrices.
For $M\in\mathbb S^n$, we write $M\succeq0$ (resp.\ $M\succ0$) if $M$ is
positive semidefinite (resp.\ positive definite). The trace inner product on
$\mathbb S^n$ is
$$
\langle U,V\rangle:=\tr(UV)\qquad (U,V\in\mathbb S^n),
$$
and the corresponding norm is
$$
\|U\|:=\sqrt{\langle U,U\rangle}.
$$

We recall some standard facts about positive semidefinite matrices that will
be used below; see, e.g., \cite[Chapters~4 and~7]{HornJohnson2013}.

\medskip
Every matrix $Z\in\mathbb S^n$ admits a spectral decomposition
$$
Z=Q\Lambda Q^\top,
$$
where $Q=[q^1,\ldots,q^n]$ is orthogonal and
$$
\Lambda=\operatorname{diag}(\mu_1,\ldots,\mu_n),
$$
with $\mu_1,\ldots,\mu_n$ being the eigenvalues of $Z$ and
$q^1,\ldots,q^n$ the corresponding orthonormal eigenvectors.

\medskip
If $Z\succeq0$, then $\mu_i\ge0$ for all $i$, and hence
$$
\tr Z=\sum_{i=1}^n\mu_i
\qquad\text{and}\qquad
\|Z\|^2=\tr(Z^2)=\sum_{i=1}^n\mu_i^2.
$$
Consequently, 
\begin{equation}\label{ine_trace_norm}
\|Z\|\le\tr Z.
\end{equation}

\medskip
Moreover, if $Z\in \mathbb{S}^n_+$ and
${\rm rank}\,Z=r$, then exactly $r$ eigenvalues of $Z$ are positive.
After reordering them if necessary, we can express
$$
Z=Q_r\Lambda_rQ_r^\top,
$$
where $Q_r=[q^1,\ldots,q^r]\in\R^{n\times r}$ has orthonormal columns and
$$
\Lambda_r=\operatorname{diag}(\mu_1,\ldots,\mu_r),
\qquad \mu_i>0.
$$
Set
$$
\Lambda_r^{1/2}
:=
\operatorname{diag}(\sqrt{\mu_1},\ldots,\sqrt{\mu_r})
$$
and
$$
V:=Q_r\Lambda_r^{1/2}.
$$
Then
\begin{equation}\label{equ_matrix_reprensent}
Z=VV^\top.
\end{equation}
Since $\Lambda_r^{1/2}$ is nonsingular and $Q_r$ has full column rank,
$V$ also has full column rank. 

It follows from  \eqref{equ_matrix_reprensent} that
\begin{equation}\label{equ_matrix_reprensent_2}
Z=\sum_{j=1}^r x^j(x^j)^\top,
\end{equation}
where $x^j:=\sqrt{\mu}q^j\,,j=1,\ldots,r$. In particular, if $\rank Z=1$ then $Z=xx^T$ for some $x\in \mathbb{R}^n$.

\medskip
We will also use the self-duality of $\mathbb S_+^n$, namely,
\begin{equation}\label{self_duality}
M\succeq0
\quad\Longleftrightarrow\quad
\langle M,X\rangle\ge0
\quad\text{for all }X\succeq0.
\end{equation}

\medskip
Let $\Omega$ be a convex subset of a finite-dimensional normed space.
A point $\bar z\in\Omega$ is called an \emph{extreme point} of $\Omega$ if,
whenever
$$
\bar z=\theta z^1+(1-\theta)z^2
$$
for some $z^1,z^2\in\Omega$ and $\theta\in(0,1)$, one has
$$
z^1=z^2=\bar z.
$$
Every nonempty compact convex set in a finite-dimensional normed space has
at least one extreme point; see, e.g.,
\cite[Corollary~18.5.1]{Rockafellar1970}.

\medskip
Let $A_1,A_2,A_3\in\mathbb S^n$ be symmetric matrices, and define the
quadratic forms
\begin{equation}\label{eq_def_fi}
f_i(x):=\tfrac12 x^\top A_i x
\qquad (i=1,2,3),\quad x\in\R^n.
\end{equation}
We consider the upper image defined by
$$
S:=\{(f_1(x),f_2(x),f_3(x)):\ x\in\R^n\}+\R_+^3.
$$

The next section studies the convexity of this set.

\section{Convexity of the Upper Image}

We first establish the following lemma, which is a key step in the proof of
our main result.

\begin{lemma}\label{lem_rankone}
Let $A_1,A_2,A_3\in\mathbb S^n$. Assume that there exists
$\lambda=(\lambda_1,\lambda_2,\lambda_3)\in\R_+^3\setminus\{0\}$ such that
\begin{equation}\label{ineq_matrix_0}
\lambda_1A_1+\lambda_2A_2+\lambda_3A_3\succeq0.
\end{equation}
Then, for every $X\succeq0$, there exists $x\in\R^n$ such that
\begin{equation}\label{ineq_matrix}
x^\top A_i x\le\langle A_i,X\rangle,\qquad i=1,2,3.
\end{equation}

\end{lemma}

\begin{proof}
	If $X=0$, then we get the conclusion by taking $x=0$. Now suppose that $X\ne0$ and set
	$$
	a_i:=\langle A_i,X\rangle,\qquad i=1,2,3.
	$$
	Consider the problem
	\begin{equation}\label{eq_trace_min}
	\min\{\tr Z:\ Z\in\mathcal P\},
	\end{equation}
	where
	$$
	\mathcal P:=
	\big\{Z\in\mathbb S_+^n:\ \langle A_i,Z\rangle\le a_i,\quad i=1,2,3\big\}.
	$$
	Since $X\in\mathcal P$, the set $\mathcal P$ is nonempty. Moreover, the set
	$	\{Z\in\mathcal P:\ \tr Z\le\tr X\}$
	is closed and bounded. Indeed, for $Z$ belonging to this set, from \eqref{ine_trace_norm} we have
	$	\|Z\|\le\tr Z\le\tr X.	$
	Hence this set is compact, and the problem \eqref{eq_trace_min} admits a solution. Moreover, it is clear that the solution set of \eqref{eq_trace_min}  is convex and compact.
So it has an extreme point. 

\medskip
Let $Z^*$ be  an extreme point of this set.
	We first note that $Z^*$ is also an extreme point of $\mathcal P$. Indeed,
	suppose that
	$$
	Z^*=\frac12(Z^1+Z^2)
	$$
	for some $Z^1,Z^2\in\mathcal P$. Since $Z^*$ is a solution and the
	objective function is linear, we have
	$$
	\tr Z^*
	=
	\frac12\big(\tr Z^1+\tr Z^2\big).
	$$
	Since
	$	\tr Z^1\ge\tr Z^*$	and	$\tr Z^2\ge\tr Z^*,	$
	both $Z^1$ and $Z^2$ are also solutions. Since $Z^*$ is an extreme point of the solution set, we get 
	$Z^1=Z^2=Z^*$, implying that $Z^*$ is an extreme point of $\mathcal{P}$.
	
		If $\rank Z^*=0$,  then we get the conclusion by taking $x=0$. If $\rank Z^*=1$, we obtain
	$Z^*=xx^\top$ for some $x\in \mathbb{R}^n$.
	Since $Z^*\in\mathcal P$,
	$$
	x^\top A_i x
	=
	\langle A_i,Z^*\rangle
	\le \langle A_i,X\rangle.
	\qquad i=1,2,3.
	$$
	This means that \eqref{ineq_matrix} holds.
	
	\medskip
We now consider the case
	$$
	r:=\rank Z^*\ge2.
	$$
	By \eqref{equ_matrix_reprensent}, we can express $Z^*$ as
	$	Z^*=VV^\top,$
	where $V\in\R^{n\times r}$ has full column rank. Let
	$$
	\mathcal I:=
	\{i\in\{1,2,3\}:\ \langle A_i,Z^*\rangle=a_i\}
	$$
	be the set of active constraints at $Z^*$.
	
	If
\begin{equation}\label{ine_dim_active}
		\dim\mathbb S^r=\frac{r(r+1)}2>|\mathcal I|,
\end{equation}
	then there exists a nonzero matrix $H\in\mathbb S^r$ such that
	$$
	\langle V^\top A_iV,H\rangle=0,
	\qquad i\in\mathcal I.
	$$
	For sufficiently small $t>0$, we have
	$	I_r\pm tH\succ0$,	where $I_r$ denotes the $r\times r$ identity matrix.
	In addition, for $i\notin\mathcal I$, we have
	$	\langle A_i,Z^*\rangle<a_i.	$
	Hence, by taking $t>0$ sufficiently small, we also have
	$$
	\big\langle A_i,V(I_r\pm tH)V^\top\big\rangle<a_i,
	\qquad i\notin\mathcal I.
	$$
	Therefore,
	$$
	Z^\pm:=V(I_r\pm tH)V^\top
	$$
	belong to $\mathcal P$. Since $H\ne0$ and $V$ has full
	column rank, $Z^+\ne Z^-$. On the other hand, it holds that
	$$
	Z^*=\frac12(Z^++Z^-).
	$$
	This contradicts the fact that $Z^*$ is an extreme point of $\mathcal P$. It follows that \eqref{ine_dim_active} does not hold, implying 
	$$
	\frac{r(r+1)}2\le 	|\mathcal I|\le3.
	$$
Combining this with the fact that $r\ge2$  gives
	$$
	r=2
	\qquad\text{and}\qquad
	|\mathcal I|=3.
	$$
	Thus all three constraints are active at $Z^*$.
	
	Set
	$$
	B_i:=V^\top A_iV,\qquad i=1,2,3.
	$$
		From the assumption \eqref{ineq_matrix_0} we obtain
		\begin{equation}\label{ineq_matrix_1}
		\sum_{j=1}^3\lambda_jB_j
		=
		V^\top
		\left(\sum_{j=1}^3\lambda_jA_j\right)V
		\succeq0.
		\end{equation}
		
	We now show that $B_1,B_2,B_3$ are linearly independent in $\mathbb S^2$. Otherwise, since
	$\dim\mathbb S^2=3$, there would exist a nonzero matrix $H\in\mathbb S^2$ such
	that
$$
	\langle B_i,H\rangle=0,
	\qquad i=1,2,3.
	$$
	
	Using the same argument as above, for sufficiently small $t>0$, set
	$$
	Z^\pm:=V(I_2\pm tH)V^\top.
	$$
	Then $Z^+,Z^-\in\mathcal P$ and $Z^+\ne Z^-$. Moreover,
	$	Z^*=\frac12(Z^++Z^-).$
	This contradicts the fact that $Z^*$ is an extreme point of $\mathcal P$.
	 Hence
	$B_1,B_2,B_3$ form a basis of $\mathbb S^2$. 
	Thus, by \cite[Theorem 14.2]{BowenWang1976} there exists a  basis
	$\{G_1,G_2,G_3\}\subset\mathbb S^2$ such that
		\begin{equation}\label{equ_basis}
	\langle B_j,G_i\rangle=\delta_{ij},
	\qquad i,j=1,2,3,
	\end{equation}
	where $\delta_{ij}=1$ if $i=j$ and $\delta_{ij}=0$ otherwise.

	We next show that at least one $G_i$ has a positive eigenvalue. Otherwise,
	$G_i\preceq0$ for every $i=1,2,3.$ Combining this with \eqref{ineq_matrix_1} and \eqref{equ_basis}, we get
	$$
	0\le\lambda_i
	=
	\Big\langle
	\sum_{j=1}^3\lambda_jB_j,G_i
	\Big\rangle
	\le0,
	\qquad i=1,2,3.
	$$
	This means that $\lambda=0$, a contradiction.
	
	\medskip
	Hence, without loss of generality, we assume that the largest eigenvalue of $G_1$
	$	\gamma:=\lambda_{\max}(G_1)>0.	$
	This implies that	$	W:=I_2-\frac1\gamma G_1\succeq0.$	
	Moreover, one eigenvalue of $W$ is zero, and so
	$	\rank W\le~1.$
	
	For $i=1,2,3$,
	$$
	\langle B_i,W\rangle
	=
	\tr B_i-\frac{\delta_{i1}}{\gamma}
	\le\tr B_i.
	$$
	Since all three constraints are active at $Z^*$, it holds
	$$
	\tr B_i
	=
	\langle A_i,Z^*\rangle
	=
	a_i.
	$$
	Thus
	$$
	\langle B_i,W\rangle\le a_i,
	\qquad i=1,2,3.
	$$
	It follows that 
		$$
	\langle A_i,\widehat Z\rangle
	=
	\langle B_i,W\rangle
	\le a_i,
	\qquad i=1,2,3,
	$$
	where	$	\widehat Z:=VWV^\top.$
	Moreover, $\widehat Z\succeq0$ as $W\succeq0$.
	Hence $\widehat Z\in\mathcal P$. In addition, since $\rank W\le 1$ we also have $\rank\widehat Z\le1$.
	
	\medskip
	Using the same argument as in case ${\rm rank}Z^*\le1$, we get the conclusion.

\medskip
	The proof is complete.
\end{proof}

We can now characterize completely the convexity of the upper image.

\medskip
\begin{theorem}\label{thm_S_convex}
Let $f_i$ $(i=1,2,3)$ be the quadratic forms defined in \eqref{eq_def_fi}. Then the upper image
$$
S=\{(f_1(x),f_2(x),f_3(x)):\ x\in\R^n\}+\R_+^3
$$
is convex if and only if at least one of the following conditions holds:
\begin{enumerate}
\item[(i)] There exists $\lambda=(\lambda_1,\lambda_2,\lambda_3)\in\R_+^3\setminus\{0\}$ such that
$$
\lambda_1A_1+\lambda_2A_2+\lambda_3A_3\succeq0.
$$
\item[(ii)] There exists $d\in\R^n$ such that
$$
d^\top A_i d<0,\qquad i=1,2,3.
$$
\end{enumerate}
Moreover, conditions \textnormal{(i)} and \textnormal{(ii)} are mutually exclusive, and under the condition \textnormal{(ii)} one has $S=\R^3$.
\end{theorem}

\begin{proof}
	First suppose that \textnormal{(i)} holds. Let
\begin{equation}\label{eq_y_j}
	y^j=(f_1(x^j),f_2(x^j),f_3(x^j))+r^j\in S,\quad j=1,2,
\end{equation}
	where $ x^j\in \mathbb{R}^n,\, r^j\in\R_+^3,$	and let $\theta\in(0,1)$. Set
	$$
	X:=\theta x^1(x^1)^\top+(1-\theta)x^2(x^2)^\top\succeq0.
	$$
	By Lemma~\ref{lem_rankone}, there exists $x\in\R^n$ such that
	$$
	x^\top A_i x\le\langle A_i,X\rangle,
	\qquad i=1,2,3.
	$$
It follows that
	$$
	f_i(x)\le
	\theta f_i(x^1)+(1-\theta)f_i(x^2),
	\qquad i=1,2,3.
	$$
Combining this with \eqref{eq_y_j} we get
	$$
	\theta y^1+(1-\theta)y^2
	=
	(f_1(x),f_2(x),f_3(x))+r
	$$
	for some $r\in\R_+^3$. Hence
	$$
	\theta y^1+(1-\theta)y^2\in S,
	$$
	and so $S$ is convex.
	
	\medskip
	Next suppose that \textnormal{(ii)} holds. Let
	$y=(y_1,y_2,y_3)\in\R^3$. Since
	$$
	d^\top A_i d<0,\qquad i=1,2,3,
	$$
	we can choose $t>0$ large enough so that
	$$
	f_i(td)=\frac{t^2}{2}d^\top A_i d\le y_i,
	\qquad i=1,2,3.
	$$
	Hence
	$$
	y=(f_1(td),f_2(td),f_3(td))+r
	$$
	for some $r\in\R_+^3$. Thus $y\in S$. Since $y\in\R^3$ was arbitrary,
	we obtain	$	S=\R^3.	$
	In particular, $S$ is convex.
	
	\medskip
	Conversely, suppose that $S$ is convex and that \textnormal{(i)} does not
	hold. We first show that there exists $X\succeq0$ such that
	$$
	\langle A_i,X\rangle<0,
	\qquad i=1,2,3.
	$$
	Suppose that no such $X$ exists. Define
	$$
	K:=
	\big\{(\langle A_1,X\rangle,\langle A_2,X\rangle,
	\langle A_3,X\rangle):X\succeq0\big\}.
	$$
	Then, we have
	$$
	K\cap(-\operatorname{int}\R_+^3)=\emptyset.
	$$
	The set $K$ is a convex cone, and $-\operatorname{int}\R_+^3$ is an open
	convex cone. By the separation theorem, there exists a nonzero
	$\lambda\in\R^3$ such that
	$$
	\langle\lambda,y\rangle\ge0
	\qquad\forall\, y\in K,
	$$
	and
	$$
	\langle\lambda,z\rangle\le0\qquad
	\forall\, z\in-\operatorname{int}\R_+^3.
	$$
	The second inequality gives
	$	\lambda\in\R_+^3\setminus\{0\}$ and 
	the first inequality gives
	$$
	\Big\langle
	\sum_{i=1}^3\lambda_iA_i,X
	\Big\rangle
	\ge0,\qquad
	\forall\, X\succeq0.
	$$
	Combining this with \eqref{self_duality}, we obtain
	$$
	\sum_{i=1}^3\lambda_iA_i\succeq0,
	$$
	a contradiction. Therefore, there exists
	$X\succeq0$ such that
	$$
	\langle A_i,X\rangle<0,
	\qquad i=1,2,3.
	$$
	
Since $X\succeq0$, by \eqref{equ_matrix_reprensent_2}, $X$ can be written as
$$
X=\sum_{j=1}^p x^j(x^j)^\top,
$$
where $p=\rank X$ and $x^1,\ldots,x^p\in\R^n$.
	Since $S$ is a convex cone, it is closed under addition. Thus
	$$
	\Big(	\frac12\langle A_1,X\rangle,
	\frac12\langle A_2,X\rangle,
	\frac12\langle A_3,X\rangle
	\Big)
=
	\sum_{j=1}^p
	(f_1(x^j),f_2(x^j),f_3(x^j))
	\in S.
	$$
Hence, there exist $d\in\R^n$ and
	$r=(r_1,r_2,r_3)\in\R_+^3$ such that
	$$
	f_i(d)+r_i
	=
	\frac12\langle A_i,X\rangle,
	\qquad i=1,2,3.
	$$
	It follows that
	$$
	d^\top A_i d
	\le
	\langle A_i,X\rangle
	<0,
	\qquad i=1,2,3.
	$$
	Thus \textnormal{(ii)} holds.
	
	Finally, conditions \textnormal{(i)} and \textnormal{(ii)} cannot hold at
	the same time. Otherwise,
	$$
	0
	\le
	d^\top\left(\sum_{i=1}^3\lambda_iA_i\right)d
	=
	\sum_{i=1}^3\lambda_i d^\top A_i d
	<0,
	$$
a contradiction.

\medskip
The proof is complete.
\end{proof}

\begin{corollary}\label{cor_one_convex}
If at least one of the matrices $A_1,A_2,A_3$ is positive semidefinite, then $S$ is convex.
\end{corollary}

\begin{proof}
	Suppose that $A_k\succeq0$ for some $k\in\{1,2,3\}$. Define
	$\lambda=(\lambda_1,\lambda_2,\lambda_3)\in\mathbb R_+^3$ by
	$$
	\lambda_k=1
	\qquad\text{and}\qquad
	\lambda_i=0\quad\text{for }i\ne k.
	$$
	Then $\lambda\ne0$ and
	$$
	\lambda_1A_1+\lambda_2A_2+\lambda_3A_3=A_k\succeq0.
	$$
	Thus condition {\rm (i)} of Theorem~\ref{thm_S_convex} holds, and hence
	$S$ is convex.
\end{proof}

\medskip
We next present several examples illustrating the characterization in Theorem~\ref{thm_S_convex} and the role of the homogeneity assumption. The first shows that the upper image may be convex even when the joint range is nonconvex.

\begin{example}\label{exam_convexcase}
	Consider the quadratic forms on $\R^3$ given by
	$$
	f_1(x)=x_1x_3+x_2x_3,\qquad 
	f_2(x)=2x_1^2-x_2^2,\qquad 
	f_3(x)=x_3^2,\qquad 
	x=(x_1,x_2,x_3)\in\R^3.
	$$
	It is shown in \cite[Example~3.2]{Polyak1998} that the joint range
	$\{(f_1(x),f_2(x),f_3(x)):\ x\in\R^3\}$ is not convex.
	Nevertheless, since $f_3$ is convex, Corollary~\ref{cor_one_convex} implies that the corresponding upper image
	$$
	S=\{(f_1(x),f_2(x),f_3(x)):\ x\in\R^3\}+\R_+^3
	$$
	is convex.
\end{example}

\begin{example}\label{exam_nonconvex}
	Define three quadratic forms on $\R^3$ by
	$$
	f_1(x)=x_1x_2,\qquad 
	f_2(x)=x_2x_3,\qquad 
	f_3(x)=x_3x_1,\qquad 
	x=(x_1,x_2,x_3)\in\R^3.
	$$
	Then
	$$
	S=\{(f_1(x),f_2(x),f_3(x)):\ x\in\R^3\}+\R_+^3
	$$
	is not convex.
\end{example}

\begin{proof}

We will show that neither condition of
Theorem~\ref{thm_S_convex} holds. Indeed, we have
$$
f_1(x)f_2(x)f_3(x)=(x_1x_2x_3)^2\ge0,
$$
so condition {\rm (ii)} does not hold. Moreover, if
$$
\lambda_1A_1+\lambda_2A_2+\lambda_3A_3\succeq0
$$
for some $\lambda=(\lambda_1,\lambda_2,\lambda_3)\in\R_+^3$, then 
$\lambda_1=\lambda_2=\lambda_3=0.$ This means that $\lambda=0$.
Thus condition {\rm (i)} also fails. By
Theorem~\ref{thm_S_convex}, $S$ is not convex.
\end{proof}

The next example shows that the homogeneity assumption in Corollary~\ref{cor_one_convex} is essential. The result does not extend to three
general quadratic functions, even when at least one component is convex.
\begin{example}
 On $\mathbb R$, we define
 $$
 f_1(t)=-2t^2+3t,\qquad
 f_2(t)=t^2-t,\qquad
 f_3(t)=t^2-3t .
 $$
 Then $f_2$ and $f_3$ are convex. However, the upper image
 $$
 \{(f_1(t),f_2(t),f_3(t)):t\in\mathbb R\}+\mathbb R^3_+
 $$
 is not convex.  
 \end{example} 
 Indeed, we have
 $$
 (f_1(-1),f_2(-1),f_3(-1))=(-5,2,4),
 \qquad
 (f_1(1),f_2(1),f_3(1))=(1,0,-2),
 $$
 so their midpoint is
 $ \bar y=(-2,1,1).$
 If $\bar y$ belongs to the upper image, then there exists $t\in\mathbb R$
 such that
 $$
 -2t^2+3t\le -2,\qquad
 t^2-t\le 1,\qquad
 t^2-3t\le 1.
 $$
 The first inequality gives $t\le -1/2$ or $t\ge2$, while the second gives
 $$
 \frac{1-\sqrt5}{2}\le t\le \frac{1+\sqrt5}{2},
 $$
 and the third gives
 $$
 \frac{3-\sqrt{13}}{2}\le t\le \frac{3+\sqrt{13}}{2}.
 $$
 These three conditions are incompatible. Hence $\bar y$ is not in the upper
 image, proving its nonconvexity.

\section{Applications}

We now present several applications of the convexity results obtained in the
previous section. We first consider an associated value-function problem.
We show that its feasible-parameter set is convex and, when the objective
quadratic form is convex, the corresponding value function is also convex.

\medskip
For $(a,b)\in\R^2$, define
$$
D(a,b):=\big\{x\in\R^n:\ f_1(x)\le a,\ f_2(x)\le b\big\},
$$
and
$$
C:=\big\{(a,b)\in\R^2:\ D(a,b)\ne\emptyset\big\}.
$$

\begin{proposition}
	The set $C$ is a convex cone. More precisely,
	\begin{equation}\label{eq_domC}
	C=\big\{(f_1(x),f_2(x)):\ x\in\R^n\big\}+\R_+^2.
	\end{equation}
\end{proposition}

\begin{proof}
	The representation \eqref{eq_domC} follows from the general formula for the
	domain of the optimal value function given in the discussion following
	equation (9) in \cite{FloresBazan2017}. 
	
	Indeed, in the notation of that
	paper, take the underlying set to be $\R^n$ and
	$$
	g(x)=(f_1(x),f_2(x)),\qquad P=\R_+^2.
	$$
	Then  the domain of the corresponding
	optimal value function is $g(\R^n)+P$, which gives exactly \eqref{eq_domC}.
	
	Moreover, by \cite[Theorem~1]{Dines1941}, the joint range
	$$
	\{(f_1(x),f_2(x)):\ x\in\R^n\}
	$$
	is convex. Since $f_1$ and $f_2$ are homogeneous quadratic forms, this set
	is also a cone. Hence \eqref{eq_domC} shows that $C$ is a convex cone.
\end{proof}

We next consider the corresponding value function. Define
\begin{equation}\label{eq_defpsi}
\psi(a,b):=\inf_{x\in D(a,b)}f_3(x),
\end{equation}
where $\inf\emptyset:=+\infty$. By the definitions of $C$ and $\psi$,
we have $C=\dom\psi.$

\begin{proposition}\label{prop_psi_convex}
	Assume that $A_3\succeq0$. Then the function $\psi$ defined in \eqref{eq_defpsi}  is finite and
	convex on $C$.
\end{proposition}

\begin{proof}
	Since $A_3\succeq0$, it follows from Corollary~\ref{cor_one_convex} that
	the upper image $S$ is convex.
Moreover, we have $f_3(x)\ge0$ for all $x\in\R^n.$ Thus $\psi$ is finite on $C$.

\medskip
	Let $z^j=(a_j,b_j)\in C$, $j=1,2$, let $\theta\in(0,1)$, and fix
	$\varepsilon>0$. By the definition of $\psi$ there
	exists $x^j\in D(z^j)$ such that
	$$
	f_3(x^j)\le \psi(z^j)+\varepsilon,
	\qquad j=1,2.
	$$
	Consequently,
	$$
	(a_j,b_j,\psi(z^j)+\varepsilon)\in S,
	\qquad j=1,2.
	$$
	Since $S$ is convex,
	$$
	\big(\theta a_1+(1-\theta)a_2,
	\theta b_1+(1-\theta)b_2,
	 \theta\psi(z^1)
	+(1-\theta)\psi(z^2)+\varepsilon\big)\in S.
		$$
Hence, there exists $x\in\R^n$ such that
\begin{equation}\label{ineq_domain_D}
	f_1(x)\le\theta a_1+(1-\theta)a_2,\qquad
	f_2(x)\le\theta b_1+(1-\theta)b_2,
\end{equation}
	and
\begin{equation}\label{ineq_domain_object}
	f_3(x)\le
	\theta\psi(z^1)+(1-\theta)\psi(z^2)+\varepsilon.
\end{equation}
From \eqref{ineq_domain_D} we obtain 
$x\in D(\theta z^1+(1-\theta)z^2)$.	From this and \eqref{ineq_domain_object} we deduce that
	$$
	\psi\big(\theta z^1+(1-\theta)z^2\big)
	\le
	\theta\psi(z^1)+(1-\theta)\psi(z^2)+\varepsilon.
	$$
Passing the latter inequality to the limit as $\varepsilon\to0$ yields
	$$
	\psi\big(\theta z^1+(1-\theta)z^2\big)
	\le
	\theta\psi(z^1)+(1-\theta)\psi(z^2),
	$$
	which proves the convexity of $\psi$ on $C$.
\end{proof}

This result shows a hidden convexity property of the associated parametric
quadratic problem. Although the constraint functions $f_1$ and $f_2$ need
not be convex, the optimal value depends convexly on the right-hand-side
parameters $(a,b)$.

\medskip
We next give an application of the first condition in Theorem~\ref{thm_S_convex}.
It leads to a generalized S-lemma for an implication with two quadratic inequalities.
  
  \begin{theorem}\label{prop_shifted_Sprocedure}
  	Let $f_i$ $(i=1,2,3)$ be the quadratic forms defined in \eqref{eq_def_fi}. Assume that there exists
  	$\mu=(\mu_1,\mu_2,\mu_3)\in\R_+^3\setminus\{0\}$ such that
  	$$
  	\mu_1A_1+\mu_2A_2+\mu_3A_3\succeq0.
  	$$
  	Fix $(a,b,c)\in\R^3$ and suppose that the Slater condition
  	\begin{equation}\label{eq_shift_slater}
  	\exists\,x^0\in\R^n:\quad f_2(x^0)<b,\qquad f_3(x^0)<c
  	\end{equation}
  	holds. Consider the implication
  	\begin{equation}\label{eq_shift_imp}
  	f_2(x)\le b,\ \ f_3(x)\le c \ \Longrightarrow\ f_1(x)\ge a.
  	\end{equation}
  	Then \eqref{eq_shift_imp} holds if and only if there exist $\lambda_1,\lambda_2\ge 0$ such that
  	\begin{equation}\label{eq_shift_cert}
  	f_1(x)-a-\lambda_1\bigl(b-f_2(x)\bigr)-\lambda_2\bigl(c-f_3(x)\bigr)\ge 0
  	\qquad \forall x\in\R^n.
  	\end{equation}
  	Equivalently,
  	$$
  	f_1(x)-a+\lambda_1\bigl(f_2(x)-b\bigr)+\lambda_2\bigl(f_3(x)-c\bigr)\ge 0
  	\qquad \forall x\in\R^n.
  $$
  \end{theorem}
  
  \begin{proof}
  	By condition \textnormal{(i)} of Theorem~\ref{thm_S_convex}, the set $S$ is convex. Hence its translation
  	$$
  	S-(a,b,c)=\{(f_1(x)-a,\ f_2(x)-b,\ f_3(x)-c):\ x\in\R^n\}+\R_+^3
  	$$
  	is also convex. Define
  	$$
  	\phi_0(x):=f_1(x)-a,\quad \phi_1(x):=b-f_2(x),\quad \phi_2(x):=c-f_3(x),\qquad x\in\R^n.
  	$$
  	Then \eqref{eq_shift_imp} is exactly the implication
  	$$
  	\phi_1(x)\ge0,\ \phi_2(x)\ge0 \ \Longrightarrow\ \phi_0(x)\ge0,
  	$$
  	while the Slater condition \eqref{eq_shift_slater} yields $\phi_1(x^0)>0$ and $\phi_2(x^0)>0$.
  	Moreover,
  	$$
  	\{(\phi_0(x),-\phi_1(x),-\phi_2(x)):\ x\in \R^n\}+\R_+^3
  	= S-(a,b,c),
  	$$
  	which is convex. Therefore, by \cite[Theorem~1]{LaraUrruty2022},
  	\eqref{eq_shift_imp} is equivalent to the existence of $\lambda_1,\lambda_2\ge0$ such that
  	$\phi_0-\lambda_1\phi_1-\lambda_2\phi_2\ge0$ on $\R^n$, which is exactly \eqref{eq_shift_cert}.
  \end{proof}

\begin{remark}\label{remark_2}
If condition \textnormal{(ii)} of Theorem~\ref{thm_S_convex} holds, then the implication \eqref{eq_shift_imp} cannot hold. Indeed, in that case, we have $$f_i(td)=\frac{t^2}{2}d^\top A_i d\to-\infty,\,t\to+\infty\, (i=1,2,3).$$ Hence for $t>0$ sufficiently large, setting $x:=td$, we get $f_2(x)\le b,\,f_3(x)\le c$, while $f_1(x)<a$.
\end{remark}

\medskip
We finish with a theorem of the alternative. Here both conditions in Theorem~\ref{thm_S_convex} have a direct role.
 
  \begin{theorem}\label{prop_strict_alt}
  	Let $f_i$ $(i=1,2,3)$ be the quadratic forms defined in \eqref{eq_def_fi}, and fix $(a,b,c)\in\R^3$.
  	Assume that at least one of the following conditions holds:
  	\begin{enumerate}
  		\item[(A)] There exists $\mu=(\mu_1,\mu_2,\mu_3)\in\R_+^3\setminus\{0\}$ such that
  		$$
  		\mu_1A_1+\mu_2A_2+\mu_3A_3\succeq0.
  		$$
  		\item[(B)] There exists $d\in\R^n$ such that
  		$$
  		d^\top A_i d<0,\qquad i=1,2,3.
  		$$
  	\end{enumerate}
  	Then exactly one of the following statements holds:
  	\begin{enumerate}
  		\item[(i)] There exists $x\in\R^n$ such that
  		$$
  		f_1(x)<a,\qquad f_2(x)<b,\qquad f_3(x)<c.
  		$$
  		\item[(ii)] There exists $\lambda=(\lambda_1,\lambda_2,\lambda_3)\in\R_+^3$ with $\lambda_1+\lambda_2+\lambda_3=1$ such that
  		$$
  		\lambda_1\bigl(f_1(x)-a\bigr)+\lambda_2\bigl(f_2(x)-b\bigr)+\lambda_3\bigl(f_3(x)-c\bigr)\ge 0
  		\qquad \forall x\in\R^n.
  		$$
  	\end{enumerate}
Moreover, under condition {\rm (B)}, statement {\rm (i)} always holds, and hence statement {\rm (ii)} does not hold.
  \end{theorem}

\begin{proof}
	Under either condition {\rm (A)} or condition {\rm (B)},
	Theorem~\ref{thm_S_convex} implies that the upper image $S$ is convex.
	Hence the translated set
	$$
	S-(a,b,c)
	=
	\{(f_1(x)-a,f_2(x)-b,f_3(x)-c):x\in\R^n\}
	+\R_+^3
	$$
	is also convex.
	
	By \cite[Proposition~2.1]{Roubi2022}, Assumption~1.1 in
	\cite[p.~1216]{Roubi2022} holds for the triple
	$$f_1-a,f_2-b,f_3-c$$ with $\mathcal K=\R^n$.
	Therefore, \cite[Theorem~3.1]{Roubi2022} shows that exactly one of
	{\rm (i)} and {\rm (ii)} holds.
	
	\medskip
If condition {\rm (B)} holds, using the same argument as in Remark \ref{remark_2}, we deduce that statement {\rm (i)} holds. Since exactly one of {\rm (i)} and
{\rm (ii)} can hold, statement {\rm (ii)} does not hold. 
\end{proof}
\section{Concluding remarks}

In this paper, we established a complete characterization of the convexity of the
upper image associated with three homogeneous quadratic forms. This result
also covers the important special case where one of the quadratic forms is
convex.

We further obtained convexity results for an associated parametric
quadratic problem and derived applications to a generalized S-lemma and a
strict theorem of the alternative. The examples show that the result applies
beyond the convexity of the joint range and also highlight the importance of
the homogeneity assumption.

\end{document}